\documentclass[12pt]{amsart}
\usepackage{amssymb}
\usepackage{times,a4wide}
\usepackage{setspace}
\newcommand{\remove}[1]{ }

\newtheorem{theorem}{Theorem}[section]
\newtheorem{proposition}[theorem]{Proposition}
\newtheorem{lemma}[theorem]{Lemma}
\newtheorem{corollary}[theorem]{Corollary}

\theoremstyle{definition}
\newtheorem{definition}[theorem]{\rm Definition}

\theoremstyle{remark}
\newtheorem{remark}[theorem]{\rm Remark}

\newcommand{\C}{\mathbb{C}}

\newcommand{\N}{\mathbb{N}}

\renewcommand{\Im}{\operatorname{Im}}

\begin{document}

\title[Series representation of mean-periodic functions]
{Representation of mean-periodic functions in series of exponential polynomials}

\author[ H. Ouerdiane, M. Ouna\"{\i}es]
{ Habib Ouerdiane and Myriam Ounaies}
\address{D\'epartement de Math\'ematique, Facult\'e des Sciences de Tunis, Universit\'e
de Tunis El Manar, Campus universitaire, 1060 Tunis, Tunisie.}
\email{habib.ouerdiane@fst.rnu.tn}

\address{Institut de Recherche Math\'ematique Avanc\'ee, Universit\'e
Louis Pasteur 7 Rue Ren\'e Des\-car\-tes, 67084 Strasbourg CEDEX,
France.} \email{ounaies@math.u-strasbg.fr}

\date{\today}



\maketitle
\begin{abstract}
Let $\theta$ be a Young function and consider the space
$\mathcal{F}_{\theta}(\C)$ of all entire functions with
$\theta$-exponential growth. In this paper, we are interested in
the solutions $f\in \mathcal{F}_{\theta}(\C)$ of the convolution
equation $T\star f=0$, called mean-periodic functions, where $T$
is in the topological dual of $\mathcal{F}_{\theta}(\C)$. We show
that each mean-periodic function can be represented in an explicit way as a 
convergent series of exponential polynomials.
\end{abstract}
\begin{section}{Introduction}

A periodic function $f$ with period $t$ may be defined in terms of
convolution equation as a function verifying
 $$(\delta_t-\delta_0) \star f=0,$$ while a function with
zero average over an interval of length $t>0$ satisfies the
convolution equation $$\mu\star f=0,$$
 where $\mu$ is defined by $\displaystyle <\mu,f>=\frac{1}{t} \int_{-t/2}^{t/2} f(x)dx$. From the observation that the second
  notion is more natural from the point of view of experimental physics, Delsartes generalized the concept of periodic functions by introducing in \cite{De} the notion of "mean-periodic" functions as the solutions of homogeneous convolution equations. 
  
In this paper, we are dealing with the problem of representing mean
periodic functions as series of exponentials polynomials.

Let us denote by $\mathcal{H}(\C)$ the space of all entire
functions on $\C$. Let $\theta$ be a Young function and $\theta^*$
its Legendre transform
  (see Definitions \ref{young} and \ref{legendre} below). The
mean-periodic functions will lie in the space
$\mathcal{F}_\theta(\C)$ of all functions $f\in \mathcal{H}(\C)$
 such that
 \begin{equation}\label{growth}
\sup_{z\in \C} \vert f(z)\vert e^{-\theta^*(m\vert z\vert)}
<\infty, \end{equation}
 for all constants $m>0$.

We will also consider the limit case where $\theta(x)=x$. In this
case, the associated conjugate function $\theta^*$ is formally
infinite. Therefore, no growth condition of the type
\eqref{growth} is involved and we put
$\mathcal{F}_\theta(\C)=\mathcal{H}(\C)$.

 We will
say that $f\in \mathcal{F}_\theta(\C)$ is a mean-periodic function
if, for a certain non zero analytic functional $T\in {\mathcal
F}_{\theta}^{\prime}(\C)$, $f$ verifies the convolution equation
\begin{equation}\label{equ}
T\star f=0.
\end{equation}
We will then say that $f$ is $T$-mean-periodic.

For example, if we denote by $\{\alpha_k\}_k$ the zeros of the
Fourier-Borel transform of $T$ and $m_k$ their order of
multiplicity, then all exponential monomials $z^j e^{\alpha_k z}$,
with $j<m_k$ are $T$-mean-periodic functions (see Lemma
\ref{monomials}). Then, every convergent series whose general term
is a linear combination of such exponential monomials is also a
$T$-mean-periodic function.

 Our main result (see Theorem \ref{gen})
states roughly that the converse holds, provided that we apply an 
Abel summation procedure in order to make the sum convergent. In fact,
we prove that any $T$-mean-periodic function $f\in
\mathcal{F}_{\theta}(\C)$
 admits the following expansion as
a convergent series in  $\mathcal{F}_\theta(\C)$
\begin{equation}\label{intro1}
f(z)=\sum_k \sum_{l=0}^{m_k-1} c_{k,l}\left[\sum_{j=0}^k
e^{z\alpha_j}P_{k,j,l}(z)\right],
\end{equation}
 where $P_{k,j,l}$ are polynomials of degree $<m_j$, explicitely given by \eqref{poly1} and \eqref{poly2}
 in terms of $\{(\alpha_k,m_k)\}_k$. Moreover, the coefficients $c_{k,l}$ verify the growth
  condition \eqref{cnl} and can be explicitly computed in terms of $f$ and
$T$.

 When $V=\{(\alpha_k,m_k)\}_k$ is an interpolating variety
 (see Definition \ref{intvar}), no Abel summation process is needed, we simply obtain that any
 $T$-mean-periodic function $f\in \mathcal{F}_{\theta}(\C)$
admits the following expansion as a convergent series in
$\mathcal{F}_\theta(\C)$
\begin{equation}\label{intro2}
f(z)=\sum_k  e^{z\alpha_k}\sum_{l=0}^{m_k-1}d_{k,l}
\frac{z^j}{j!} ,
\end{equation}
where the coefficients $d_{k,l}$ verify the growth estimate
\eqref{dkj} (see Theorem \ref{int}).

Our present work is inspired by the paper \cite{Be-Ta},  written by C.A. Berenstein and B.A. Taylor in 1975
where the authors considered the case
$\mathcal{F}_{\theta}(\C)=\mathcal{H}(\C)$. In fact, they showed
that, given $T\in {\mathcal H}^{\prime}(\C)$, there exists a
sequence of indices $k_0=0<k_1<\cdots<k_n<\cdots$  such that any $T$-mean
periodic function $f\in \mathcal{H}(\C)$ admits a unique
expansion, convergent in $ \mathcal{H}(\C)$, of the form
\begin{equation}\label{intro3}f(z)=\sum_n\sum_{k_n\le k<k_{n+1}}
e^{z\alpha_k }\sum_{j=0}^{m_k-1}d_{k,j} \frac{z^j}{j!}.
\end{equation}

 In \eqref{intro3},
the sum converges by packets grouping instead of Abel summation. But in general, the sequence $\{k_n\}_n$ is not
explicit, except in the case when $V$ is an interpolating variety,
where the sequence $k_n=n$ works, thus formula \eqref{intro3}
leads to \eqref{intro2}.

In 1988,  representation formulas of the form \eqref{intro3}  were given by C.A. Berenstein and D.C. Struppa  (cf \cite{Be-St}) in the case  where $\theta(x)=x^p$, $p>1$ and in the more complicated situation where $\mathcal{F}_\theta(\C)$ is replaced by $\mathcal{F}_\theta(\Gamma)$ with $\Gamma$ an open convex cone in $\C$, provided some natural conditions on the behavior of the Fourier-Borel transform of $T$. 

We also refer the interested reader to \cite{Be-St1} for a general survey on the connections between mean periodicity and complex anlaysis in $\C^n$. 

To conclude the introduction, here is how the paper is organized : Section \ref{prel} is devoted to
preliminary definitions and useful results from functional
analysis. The main results are stated in section \ref{main} and their proofs
are given  in section \ref{proof}. Finally, in section \ref{intcase}, we study
the particular case when $V$ is an interpolating variety.

 \end{section}

\begin{section}{Preliminaries and definitions.}\label{prel}

\begin{definition}\label{young}
A function $\theta : [0,+\infty[\rightarrow [0,+\infty[$ is called
a Young function if it is convex, continuous, increasing and
verifies $\theta(0)=0$ and $r=o(\theta(r))$ when $r\rightarrow
+\infty$.
\end{definition}

\begin{definition}\label{legendre}
Let $\theta$ be a Young function. The Legendre transform
$\theta^*$ of $\theta$ is the function defined by
$${\theta^*}(x)=\sup_{t\ge 0} (tx-\theta(t)).$$
\end{definition}

Note that the Legendre transform of a Young function is itself a Young
function and $\theta^{**}=\theta$. We refer the reader to
\cite{Kr} for further details.
Throughout the paper, $\theta$ will denote either the function
$\theta(x)=x$ or a Young function.

For any $m>0$, consider $E_{\theta,m}(\C)$, the Banach space of
all functions $f\in \mathcal{H}(\C)$ such that $$\Vert
f\Vert_{\theta,m}:=\sup_{z\in \C} \vert f(z)\vert
e^{-\theta(m\vert z\vert)}<+\infty$$ and define
$$\mathcal{G}_{\theta}(\C)=\cup_{p\in \N^*}E_{\theta,p}(\C)$$
endowed with the inductive limit topology. It is clear that
$\mathcal{G}_{\theta}(\C)$ is an algebra under the ordinary
multiplication of functions.

\begin{remark}
When $\theta(x)=x^k$, $k\ge 1$, $\mathcal{G}_{\theta}(\C)$ is the
space of all entire functions, either of order $< k$ or of order
$k$ and finite type. In particular, when $k=1$,
$\mathcal{G}_{\theta}(\C)$ is the space of all entire functions of
exponential type, usually denoted by  Exp$(\C)$.
\end{remark}
We define the space $\mathcal{F}_{\theta}(\C)$ as follows :

(i) In the case where $\theta(x)=x$, we put
$\mathcal{F}_{\theta}(\C)=\mathcal{H}(\C)$, the space of all
entire functions endowed with the topology of uniform convergence
on every compact of $\C$. It is a Fr\'echet-Schwartz space (see
\cite{Be-Ga}).

(ii) In the case where $\theta$ is a Young function, we denote
$$\mathcal{F}_{\theta}(\C)=\cap_{p\in\N^*}E_{\theta^*,1/p}(\C)$$endowed
with the projective limit topology. The space
$\mathcal{F}_\theta(\C)$ is a nuclear Fr\'echet space (see
\cite{GHOR}), hence it is a Fr\'echet-Schwartz space  (see
\cite{Sc}).

For any fixed $\xi\in \C$, and $l\in \N$, we will denote by
$M_{l,\xi}$ the exponential monomial $z \rightarrow z^le^{\xi z}$.
It is easy to see that $M_{l,\xi}\in \mathcal{F}_{\theta}(\C)$. In
the next we denote by ${\mathcal F}_{\theta}^{\prime}(\C)$ the
strong topological dual of $\mathcal{F}_{\theta}(\C)$.

Let us recall some definitions and properties from functional
analysis. We refer  to \cite{Be-Ga} for further details in the
case (i) and to \cite{GHOR} for the case (ii).

To any fixed $u\in \C$, define the translation operator $\tau_u$
on $\mathcal{F}_\theta(\C)$ by $$(\tau_u f)(z)=f(z+u), \hbox{for
all } f\in \mathcal{F}_{\theta}(\C) \ \hbox{and } z\in \C.$$ It's
easy to see that  $\mathcal{F}_\theta(\C)$ is invariant under
these translation operators.

For all $S\in  {\mathcal F}_{\theta}^{\prime}(\C)$ and $f\in
\mathcal{F}_\theta(\C)$, the function $z\rightarrow <S,\tau_z f>$,
where $<\ ,\ >$ denotes the duality bracket, is an element of
$\mathcal{F}_\theta(\C)$. Therefore, for any $S\in  {\mathcal
F}_{\theta}^{\prime}(\C)$, the map $S\star :
\mathcal{F}_\theta(\C)\rightarrow \mathcal{F}_\theta(\C)$  defined
by $$S\star f(z)=<S,\tau_z f>$$
 is a convolution operator, i.e., it is linear, continuous and commute with any translation operator.

For any $S\in {\mathcal F}_{\theta}^{\prime}(\C)$, the
Fourier-Borel transform of $S$, denoted by $\mathcal{L}(S)$ is
defined by $$\mathcal{L}(S)(\xi)=<S,e^{\xi \cdot}>,$$ where
$e^{\xi \cdot}=M_{0,\xi}$ is the function $z\in \C \rightarrow
e^{\xi z}$.

For any two elements $S$ and $U$ of ${\mathcal
F}_{\theta}^{\prime}(\C)$, the convolution product $S\star U\in
{\mathcal F}_{\theta}^{\prime}(\C)$ is defined by $$\forall  f\in
\mathcal{F}_\theta(\C), \ \ <S\star U,f>=<S,U\star f>.$$ Moreover
for any $S,U \in {\mathcal F}_{\theta}^{\prime}(\C) $
$$\mathcal{L}(S\star U)=\mathcal{L}(S)\mathcal{L}(U) $$

Under this convolution, ${\mathcal F}_{\theta}^{\prime}(\C)$ is a
commutative algebra admitting $\delta_0$, the  Dirac measure at
the origin, as unit.

\begin{proposition}\label{laplace}

The Fourier-Borel transform $\mathcal{L}$ is a topological
isomorphism between the algebras ${\mathcal
F}_{\theta}^{\prime}(\C) $ and $\mathcal{G}_{\theta}(\C)$.

\end{proposition}

\end{section}

\begin{section}{Main results}\label{main}

Throughout the rest of the paper, let $T$ be a fixed non-zero
element of ${\mathcal F}^{\prime}_{\theta}(\C)$. Our main goal in
this section is to show that any function $f\in
\mathcal{F}_{\theta}(\C)$ satisfying the equation
\begin{equation}\label{equ}
T \star f=0
\end{equation}
can be represented as convergent series of exponential-polynomials
which are them-selves solution of  (\ref{equ}).

\begin{definition}
We say that a function $f\in \mathcal{F}_{\theta}(\C)$ is
"$T$-mean-periodic" if it satisfies the equation \eqref{equ}.
\end{definition}
Denote by $\Phi$ the entire function in $\mathcal{G}_{\theta}(\C)$
defined by $\Phi=\mathcal{L}(T)$. Before going further, let us
show the following division property :
\begin{lemma}\label{div}
Let $h\in \mathcal{H}(\C)$ and $g\in \mathcal{G}_{\theta}(\C)$. If
$g$ is not identically zero and if $f=gh \in
\mathcal{G}_{\theta}(\C)$, then $h\in  \mathcal{G}_{\theta}(\C)$.
\end{lemma}

\begin{proof}
Up to a translation, we may assume that $g(0)\not=0$. Let us apply
the minimum modulus theorem and it's corollary given in
\cite[Lemma 2.2.11]{Be-Ga} to the function $g$ in the disc of
center $0$ and radius $2^{n+1}e$, where $n$ is any positive
integer.

As $g\in \mathcal{G}_{\theta}(\C)$, there exists $p\in \N^*$ and
$C_p>0$ (not depending on $n$) such that $$\max_{\vert \xi \vert
\le 2^{n+3}e} \vert g(\xi )\vert  \le C_p e^{{\theta}(p 2^n)}.$$
Thus, there exists $\varepsilon_p>0$ (not depending on $n$) and
$R_n$, $2^n\le R_n\le 2^{n+1}$ such that $$\min_{\vert \xi
\vert=R_n} \vert g(\xi )\vert \ge \varepsilon_p e^{-{\theta}(p
2^n)}.$$ Let $n\in \N$ and $\vert \xi\vert =R_n$. As $f\in
\mathcal{G}_{\theta}(\C)$, there exists  $q>0$ and $C_{q}>0$ (not
depending on $n$), such that $$\vert f(\xi)\vert \le
C_{q}e^{{\theta}(q 2^n)}.$$ Using the convexity of $\theta$ and
the fact that $\theta(0)=0$, we have $$\theta( p
2^n)\le \frac{1}{2}\theta(p
2^{n+1}).$$ If
we assume, for example, that $p\ge q$, we deduce that $$\vert
h(\xi)\vert = \vert f(\xi)\vert \frac{1}{\vert g(\xi)\vert} \le
\frac{C_q}{\varepsilon_{p}} e^{\theta(q 2^n)+\theta(p 2^n)}\le B_p
e^{\theta(p 2^{n+1})}.$$

Now let $z\in \C$, be such that $2^{n-1}\le \vert z\vert \le
2^n\le R_n$. By the maximum modulus theorem, $$\vert h(z)\vert \le
B_p e^{{\theta}(p 2^{n+1})}\le B_p e^{{\theta}(4p \vert
z\vert)}.$$ This proves that $h \in \mathcal{G}_{\theta}(\C)$.
\end{proof}

\begin{corollary}
In the case where $\Phi$ has no zeros, the only mean-periodic
function $f\in \mathcal{F}_{\theta}(\C)$ is the zero function.
\end{corollary}
\begin{proof}
Assume that $\Phi$ has no zeros. Then $\frac{1}{\Phi} \in
\mathcal{H}(\C)$ and, by Lemma \ref{div}, $\frac{1}{\Phi}\in
\mathcal{G}_{\theta}(\C)$. By Proposition \ref{laplace},
$S=(\mathcal{L})^{-1}(\frac{1}{\Phi})\in {\mathcal
F}_\theta^{\prime}(\C)$. Then, we have $S\star T=T\star
S=\delta_0$. If we assume $T\star f=0$, then $\delta_0 \star
f=f=0.$
\end{proof}

We will throughout the rest of the paper assume that $\Phi$ has
zeros, and denote them by $\vert \alpha_0\vert\le \vert
\alpha_1\vert \le \cdots \le\vert \alpha_k\vert\le \cdots$,
$\alpha_k\not=\alpha_{k'}$ if $k\not=k'$.

We will denote by  $m_k$ be the order of multiplicity of $\Phi$ at
$\alpha_k$ and we will consider the "multiplicity variety"
$V=\{(\alpha_k, m_k)\}_{k\in \N}$ (see \cite{Be-Ga1} for an
introduction to the concept of multiplicity variety). We may
use the notation $V=\Psi^{-1}(0)$.

\begin{lemma}\label{monomials}

(i)  For all $\xi\in \C$ and $l\in \N$, we have
$<T,M_{l,\xi}>=\Phi^{(l)}(\xi)$.

(ii) Each exponential monomial $M_{l,\alpha_k}$, for $0\le l< m_k$
is $T$-mean-periodic.

\end{lemma}

\begin{proof}

To prove (i), we proceed by induction on $l\ge 0$. The property is
true for $l=0$ by definition of the Fourier-Borel transform of
$T$.

Suppose the property true for $l$. Let $\xi\in \C$ be fixed. Let
us verify that the function $\displaystyle
\frac{M_{l,\xi+u}-M_{l,\xi}}{u}$ converges in
$\mathcal{F}_\theta(\C)$ to $M_{l+1,\xi}$ when $u$ tends to $0$.
For all $z\in \C$ and $u\le 1$, we have
\[
\left\vert \frac{e^{uz}-1}{u}-z\right\vert=\left\vert uz^2
\sum_{n\ge 2}\frac{(uz)^{n-2}}{n!}\right\vert \le\vert u\vert
\vert z\vert^2 e^{\vert z\vert}.
\]
This implies that
\[
\left\vert
\frac{M_{l,\xi+u}(z)-M_{l,\xi}(z)}{u}-M_{l+1,\xi}(z)\right \vert
=\vert z^le^{\xi z}\vert \left\vert
\frac{e^{uz}-1}{u}-z\right\vert \le \vert u\vert \vert
z\vert^{l+2} e^{(1+\vert \xi\vert)\vert z\vert}.
\]
Therefore, $\frac{M_{l,\xi+u}-M_{l,\xi}}{u}$ converges to
$M_{l+1,\xi}$ for the topology of $\mathcal{F}_\theta(\C)$ when
$u$ tends to $0$.

From this, we obtain
\[
\begin{split}
\Phi^{(l+1)}(\xi) & =\lim_{u\rightarrow 0}
\frac{\Phi^{(l)}(\xi+u)-\Phi^{(l)}(\xi)}{u} = \lim_{u\rightarrow
0} \frac{<T,M_{l,\xi+u}>-<T,M_{l,\xi}>}{u}\\ & =
\lim_{u\rightarrow 0} <T, \frac{M_{l,\xi+u}-M_{l,\xi}}{u}> = <T,
M_{l+1,\xi}>,
\end{split}
\]
by continuity of $T$. This completes the proof of (i).

In order to prove (ii), it is sufficient to see that $$T\star
M_{l,\alpha_k} (z)=<T, \tau_z M_{l,\alpha_k}>=e^{\alpha_k
z}\sum_{n=0}^l C_l^n z^{l-n} <T, M_{n,\alpha_k}>= e^{\alpha_k
z}\sum_{n=0}^l C_l^n z^{l-n} \Phi^{(n)}(\alpha_k).$$

\end{proof}

Our main theorem states roughly that $T$-mean-periodic functions
are series of linear combinations of the exponential monomials
$M_{l,\alpha_k}$ :

\begin{theorem}\label{gen}

(i) Any $T$-mean-periodic function $f\in \mathcal{F}_{\theta}(\C)$
 admits the following expansion as
a convergent series in  $\mathcal{F}_\theta(\C)$
\begin{equation}\label{expgen}
f(z)=\sum_{k\ge 0} \sum_{l=0}^{m_k-1} c_{k,l}\left[\sum_{j=0}^k
e^{z\alpha_j}P_{k,j,l}(z)\right],
\end{equation}
 where $P_{k,j,l}$ are the polynomials of degree $<m_j$ given by (\ref{poly1}) and (\ref{poly2}).
  The coefficients $c_{k,l}$ verify the following estimate
\begin{equation}\label{cnl}
\forall m>0,\ \ \ \ \sum_{k\ge 0} e^{{\theta}(m\vert
\alpha_k\vert)}\left(\sum_{l=0}^{m_k-1} \vert c_{k,l}\vert (\vert
\alpha_k\vert+1)^{-(m_0+\cdots+m_{k-1}+l)}\right)<+\infty
\end{equation}
and are given by
 $$c_{k,l}=<S_{k,l},f>$$
  where $S_{k,l}\in {\mathcal
F}_{\theta}^{\prime}(\C)$ is defined by
$$\mathcal{L}(S_{k,l})(\xi)=(\xi-\alpha_k)^l\prod_{n=0}^{k-1}
(\xi-\alpha_n)^{m_n}.$$

\noindent (ii) Conversely, any such serie whose coefficients
$c_{n,l}$ satisfy the estimate \eqref{cnl} converges in
$\mathcal{F}_\theta(\C)$ to a function $f$ solving the equation
\eqref{equ}.

\end{theorem}

\begin{corollary}\label{simple}
Assume that all the multiplicities $m_k$ are equal to $1$. Then

\noindent (i) any $T$-mean-periodic function $f\in
\mathcal{F}_{\theta}(\C)$ admits the following expansion as a
 convergent series in  $\mathcal{F}_\theta(\C)$
\begin{equation}\label{expint}
f(z)=\sum_{k\ge 0}  c_k \left[ \sum_{j=0}^k e^{z\alpha_j}
\prod_{0\le n\le k, n\not=j}(\alpha_j-\alpha_n)^{-1} \right],
\end{equation}
where the coefficients $c_k$ satisfy the following estimate
\begin{equation}\label{cn}
\forall m>0, \ \ \ \sum_{k\ge 0} e^{{\theta}(m\vert
\alpha_k\vert)}\vert c_k\vert (\vert \alpha_k\vert+1)^{-k}<+\infty
\end{equation}

and are given by
 $$c_k=<S_k,f>$$ where $S_k \in {\mathcal
F}_{\theta}^{\prime}(\C)$ is defined by
$$\mathcal{L}(S_k)(\xi)=\prod_{n=0}^{k-1} (\xi-\alpha_n)^{m_n}.$$

\noindent (ii) Conversely, any such series whose coefficients
$c_k$ satisfy the estimate
 \eqref{cn}
 converges  in
$\mathcal{F}_\theta(\C)$ to a function $f$ solving the equation
\eqref{equ}.

\end{corollary}

\end{section}

\begin{section}{Proof of Theorem \ref{gen}}\label{proof}

We define the restriction operator $\rho$  on
$\mathcal{G}_{\theta}(\C)$ by
\[
\rho(g)=\{\frac{g^{l}(\alpha_k)}{l!}\}_{k,0\le l< m_k}, \ \
g\in \mathcal{G}_{\theta}(\C).
\]
As an immediate consequence of Lemma \ref{div}, we have the
following lemma.
\begin{lemma}\label{Ker}
The kernel of the restriction operator $\rho$ is the ideal
generated by $\Phi$ in  $\mathcal{G}_{\theta}(\C)$, i.e.,
$$\hbox{Ker }\rho=\{\Phi g, \ g\in \mathcal{G}_{\theta}(\C)\}.$$
\end{lemma}

We are going to use a characterization, obtained in \cite{Ou2}, of
the elements $a=\{a_{k,l}\}_{k,0\le l<m_{k-1}}$ belonging
to $\rho(\mathcal{G}_{\theta}(\C))$. This characterization is
given in terms of growth conditions involving the divided
differences (see \cite{Is-Ke} for further details about divided
differences).

To any discrete doubly indexed sequence $a=\{a_{k,l}\}_{k\in
\N,0\le l<m_k}$ of complex numbers, we associate  the sequence
of divided differences $\Psi(a)=\{b_{k,l}\}_{k\in \N,0\le l< m_k}$. We
recall that they are the coefficients of the Newton polynomials,
\begin{equation}\label{newton}
Q_q(\xi)=\sum_{k=0}^q \prod_{n=0}^{k-1} (\xi-\alpha_n)^{m_n}
\left(\sum_{l=0}^{m_k-1} b_{k,l} (\xi-\alpha_k)^l\right),
\end{equation}
defined, for any $q\in \N$, as the unique polynomial of degree
$m_0+\cdots+ m_q-1$  such that
$$\frac{Q_q^{(l)}(\alpha_k)}{l!}=a_{k,l},\ \ \hbox{for }\ 0\le
k\le q\ \hbox{and }\ 0\le l \le m_k-1.$$
When all the multiplicities $m_k=1$, we may give a simple formula
for the coefficients $b_k$ : $$b_k=\sum_{j=0}^k a_j \prod_{0\le
n\le k, n\not=j}(\alpha_j-\alpha_n)^{-1}.$$

\noindent In the general case (see \cite{Ou2}) we may define them  by induction by :
$$b_{0,l}=a_{0,l}, \hbox{for all}\ \  0\le l< m_0, $$
and for $k\ge 1$, 
$$b_{k,0}=\frac{a_{k,0}-Q_{k-1}(\alpha_k)}{\Pi_{k-1}(\alpha_k)},$$
$$b_{k,l}=\frac{a_{k,l}-\frac{Q_{k-1}^{(l)}(\alpha_k)}{l!}-\sum_{n=0}^{l-1}\frac{1}{(l-n)!}
\Pi_{k-1}^{(l-k)}(\alpha_j)b_{k,n}}{\Pi_{k-1}(\alpha_k)}\
\ \hbox{ for } 1\le l< m_k$$ where we have denoted by
$$\Pi_k(\xi)=\prod_{n=0}^k (\xi-\alpha_n)^{m_n}.$$

The following lemma describes the image of the map $\rho$ and is crucial for the rest of the proof. It is an easy consequence of \cite[Theorem 1.11]{Ou2}.
\begin{lemma}\label{surj}
A doubly indexed sequence $a=\{a_{k,l}\}_{k\in
\N,0\le l<m_k}$ belongs  to $\rho(\mathcal{G}_{\theta}(\C))$ if and only if 
 $$\sup_{k\in \N}  \sup_{0\le l< m_k}\vert b_{k,l}\vert (\vert
\alpha_k\vert+1)^{m_0+\cdots+m_{k-1}+l}e^{-{\theta}(m\vert
\alpha_k\vert)}<+\infty,$$ for a certain $m>0$, where
$b=\{b_{k,l}\}_{k,0\le l< m_k}=\Psi^{-1}(a)$.
\end{lemma}

In order to give a  topological structure, let us
denote by $\mathcal{B}_{\theta,m}(V)$  the Banach space of all
doubly indexed sequences of complex numbers $b=\{b_{k,l}\}_{k\in
\N, 0\le l<m_k}$  such that $$\Vert b\Vert_{\theta,m}=\sup_{k\in
\N}  \sup_{0\le l< m_k}\vert b_{k,l}\vert (\vert
\alpha_k\vert+1)^{m_0+\cdots+m_{k-1}+l}e^{-{\theta}(m\vert
\alpha_k\vert)}<+\infty.$$ Let us consider the space
$\mathcal{A}_{\theta,m}(V)=\Psi^{-1}(\mathcal{B}_{\theta,m}(V))$,
that is, the space of all doubly indexed sequences of complex
numbers $a=\{a_{k,l}\}_{k\in \N, 0\le l<m_k}$  such that $$\Vert
\Psi(a) \Vert_{\theta,m}<+\infty.$$
 It is easy to see that $\mathcal{A}_{\theta,m}(V)$ endowed with the norm
 $\Vert a\Vert_{\theta,m} =\Vert \Psi(a) \Vert_{\theta,m}$ is a Banach space and that $\Psi$ is an isometry  from $\mathcal{A}_ {\theta,m}(V)$ into
 $\mathcal{B}_{\theta,m}(V)$.
Now, we define the spaces $$\mathcal A_{\theta}(V)=\cup_{p\in
\N^*}\mathcal{A}_{\theta,p}(V) \hbox{ and }
\mathcal{B}_{\theta}(V)=\cup_{p\in N^*}\mathcal{B}_{\theta,p}(V)$$
endowed with the topology of inductive limit of Banach spaces.

\noindent We define the linear map $$\alpha= \Psi\circ \rho \circ
\mathcal{L} : \mathcal{F}_{\theta}^{\prime}(\C)\rightarrow
\mathcal{B}_{\theta}(V).$$

\begin{proposition}\label{alpha}
The map $\alpha$  is continuous and surjective.

\end{proposition}

\begin{proof}
By Proposition \ref{laplace}, we know that
$\mathcal{L}:\mathcal{F}_{\theta}^{\prime}(\C)\rightarrow
\mathcal{G}_{\theta}(\C)$ is a topological isomorphism.
By Lemma \ref{surj}, the operator $$\rho \ :\  {\mathcal
G}_{\theta}(\C)\rightarrow \mathcal{A}_{\theta}(V)$$ is surjective. It is also continuous
by \cite[Proposition 1.8]{Ou2}.
Finally, by construction, it is  clear that $$\Psi \ :\ \mathcal
A_{\theta}(V)\rightarrow {\mathcal B}_{\theta}(V)$$ is a
topological isomorphism.

\end{proof}

Recall that  $\mathcal{F}_{\theta}(\C)$ is a Fr\'echet-Schwartz
space, therefore it is reflexive. Then, the transpose $\alpha^t$
of $\alpha$ is defined from the strong dual of ${\mathcal
B}_{\theta}(V)$, denoted by ${\mathcal B}_{\theta}^{\prime}(V)$,
into $\mathcal{F}_{\theta}(\C)$.

Next, we need to characterize the dual space ${\mathcal
B}_{\theta}^{\prime}(V)$ as a space of doubly indexed sequences :

\begin{lemma}\label {dualgen}

 The space
${\mathcal B}_{\theta}^{\prime}(V)$ is topologically isomorphic through the canonical bilinear form
$$<c,b>=\sum_{k=0}^{+\infty}\sum_{l=0}^{m_k-1} c_{k,l} b_{k,l}$$to
the space $\mathcal{C}_{\theta}(V)=\cap_{p\in
\N^*}\mathcal{C}_{\theta,p}(V)$ endowed with the projective limit
topology, where, for all $p$, $\mathcal{C}_{\theta,p}(V)$ is the
Banach space of the sequences $c=\{c_{k,l}\}_{k\in \N, 0\le l<m_k}$ such
that
\begin{equation}\label{cklp}
\Vert c\Vert'_{\theta,p} :=\sum_{k\in \N} e^{{\theta}(p\vert
\alpha_k\vert)}\left(\sum_{l=0}^{m_k-1} \vert c_{k,l}\vert (\vert
\alpha_k\vert+1)^{-(m_0+\cdots+m_{k-1}+l)}\right)<+\infty.
\end{equation}
Moreover, $\mathcal{C}_{\theta}(V)$ is a Fr\'echet-Schwartz space.

\end{lemma}

\begin{proof}

 Let us
show that $\beta: \mathcal{C}_{\theta}(V)\rightarrow {\mathcal
B}_{\theta}^{\prime}(V)$ defined  by
$$<\beta(c),b>=<c,b>=\sum_{k=0}^{+\infty}\sum_{l=0}^{m_k-1} c_{k,l}
b_{k,l}$$ is a topological isomorphism.

Let  $c=\{c_{k,l}\}_{k,
0\le l<m_k}$ be an element of $\mathcal{C}_{\theta}(V)$ and
$b=\{b_{k,l}\}_{k, 0\le l<m_k} \in \mathcal{B}_{\theta,p}(V)$, for
a certain $p$. For any $k\ge 0$, we have, by definition of $
\Vert b\Vert_{\theta,p}$,
\[
 \sum_{l=0}^{m_k-1}\vert b_{k,l} c_{k,l}\vert \le e^{\theta(p \vert \alpha_k\vert)} \Vert b\Vert_{\theta,p} \sum_{l=0}^{m_k-1}\vert c_{k,l}\vert (\vert \alpha_k\vert+1)^{-(m_0+\cdots+m_{k-1}+l)}.
\]
Using the estimate \eqref{cklp}, we see that the sum converges
(absolutely) and that
 \[  \vert <c,b>\vert \le \Vert c\Vert'_{\theta,p} \Vert b\Vert_{\theta,p}.
 \]
This shows the continuity of $\beta$. Let $B^{k,l}$ be the doubly
indexed sequence of $\C$ defined by (using the Kronecker symbols)
:
\begin{equation}\label{Bnl}
B^{k,l}=\{\delta_{kj}\delta_{ln}\}_{j,0\le n<m_j}.
\end{equation}
We easily see that $B^{k,l}\in \mathcal{B}_{\theta,p}(V)$.
For all $k$ and $0\le l<m_k$, we have $c_{k,l}=<\beta(c),B^{k,l}>$. It is then clear that $\beta$ is injective.
 
 Conversely, to an element
$\nu \in {\mathcal B}_\theta^{\prime}(V)$, consider the doubly
indexed sequence $c=\{c_{k,l}\}_{k, 0\le l<m_k}$ defined by
$$c_{k,l}=<\nu, B^{k,l}>$$
To verify that $c\in \mathcal{C}_{\theta}(V)$, let $p\in \N^*$ be fixed and
define $\tilde b=\{\tilde b_{k,l}\}_{k, 0\le l<m_k}$ by
$$\tilde b_{k,l}=e^{{\theta}(p\vert \alpha_k\vert)}\frac{\bar
c_{k,l}}{\vert c_{k,l}\vert}(\vert
\alpha_k\vert+1)^{-(m_0+\cdots+m_{k-1}+l)} \ \hbox{if }
c_{k,l}\not=0, \ \ \tilde b_{k,l}=0 \ \hbox{otherwise}.$$ It is clear
that $\tilde b\in \mathcal{B}_{\theta,p}(V)$ and that $\Vert
\tilde b\Vert_{\theta,p}\le 1.$
Therefore, all the finite sequences $\tilde b^{K}=\sum_{k=0}^K
\sum_{l=0}^{m_k-1} \tilde b_{k,l}B^{k,l}$ satisfy $$\Vert
\tilde b^{K}\Vert_{\theta,p}\le 1.$$ Denoting by $\Vert \nu\Vert'_{\theta,p}$ the
 norm in ${\mathcal B}_{\theta,p}^{\prime}(V)$,  we have, for all $K$, $$\vert <\nu, \tilde b^{K}>\vert
\le \Vert \nu\Vert'_{\theta,p} \Vert \tilde b^{K}\Vert_{\theta,p}\le \Vert
\nu\Vert'_{\theta,p}.$$ 
On the other hand, $$<\nu, \tilde b^{K}>=\sum_{k=0}^K \sum
_{l=0}^{m_k-1} \tilde b_{k,l}<\nu, B^{k,l}>=\sum_{k=0}^K
e^{{\theta}(p\vert \alpha_k\vert)}\sum _{l=0}^{m_k-1} \vert
c_{k,l}\vert(\vert \alpha_k\vert+1)^{-(m_0+\cdots+m_{k-1}+l)},$$ by
definition of $\tilde b_{k,l}$. Letting $K$ tend to infinity, we obtain
that $c\in \mathcal{C}_{\theta,p}(V)$ and that
\begin{equation}\label{beta-1}
\Vert c\Vert'_{\theta,p}\le \Vert \nu\Vert'_{\theta,p}.
\end{equation}
Consider now an element $b=\{b_{k,l}\}_{k, 0\le l<m_k}$ of $\mathcal{B}_{\theta,p}(V)$ and put 
$$b^K=\sum_{k=0}^K \sum_{l=0}^{m_k-1} b_{k,l}B^{k,l}.$$  Let $q$ be an integer strictly larger than $p$. Note that by convexity of $\theta$, for all $k$ the following inequality holds
\begin{equation}\label{convexity}
-\theta(q\vert\alpha_k\vert)+\theta(p\vert \alpha_k\vert)\le -(1-p/q)\theta(q\vert\alpha_k\vert).
\end{equation}
Using this inequality, we find
$$\Vert b-b^{K}\Vert_{\theta,q}\le \Vert b\Vert_{\theta,p}\sup_{k>K} e^{-\theta(q\vert\alpha_k\vert)+\theta(p\vert \alpha_k\vert)}\le \Vert b\Vert_{\theta,p} e^{-(1-p/q)\theta(q\vert\alpha_K\vert)}.$$
We readily deduce that $b^K$ converges to $b$ when $K$ tends towards infinity and that
$$<\nu,b>=\sum_{k=0}^{+\infty} \sum_{l=0}^{m_k-1} b_{k,l}<\nu,B^{k,l}>=<c,b>.$$
We conclude that $\beta(c)=\nu$ and that $\beta$ is surjective.
The continuity of $\beta^{-1}$
is a direct consequence of the inequality (\ref{beta-1}).

In order to prove that  $\mathcal{C}_{\theta}(V)$ is a
Fr\'echet-Schwartz space, in view of \cite[Proposition
1.4.8.]{Be-Ga}, it is sufficient to see that, for any $p\in \N^*$,
the canonical injection $$i_p : \mathcal{C}_{\theta,p+1}(V)
\rightarrow \mathcal{C}_{\theta,p}(V)$$ is compact.
 Let $\{c^n\}_n$ be a sequence of elements in $\mathcal{C}_{\theta,p+1}(V)$ such that, for all $n$, $\Vert c^n\Vert_{\theta,p+1}\le 1$.
It suffices to show that one can extract a subsequence of
$\{c^n\}_n$ converging in $\mathcal{C}_{\theta,p}(V)$.

It is easy to see that, for all $k\in \N$ and $0\le l<m_k$ the
sequence $\{c^n_{k,l}\}_n$ is bounded. Thus, up to taking
a subsequence, we may assume  that
$c^n_{k,l}$ converges to a certain $c_{k,l}\in \C$. Putting
$c=\{c_{k,l}\}_{k,0\le l<m_k}$, we readily see that $c\in
\mathcal{C}_{\theta,p+1}(V)$ and $\Vert c \Vert_{\theta,p+1}\le
1$.

Let us verify that $\Vert c^n-c\Vert_{\theta,p}$ tends to $0$ when
$n$ tends to infinity.
We assume that $\vert \alpha_k\vert \rightarrow \infty$,
otherwise, the result is trivial. Then, again using inequality (\ref{convexity}) we find that
$e^{\theta(p\vert \alpha_k\vert)-\theta((p+1)\vert
\alpha_k\vert)}$ tends to $0$ when $k$ tends towards infinity. Let
$\varepsilon>0$. For a certain $K\in \N$ and for all $k\ge K$,
$e^{\theta(p\vert \alpha_k\vert)-\theta((p+1)\vert
\alpha_k\vert)}<\frac{\varepsilon}{4}$. Thus, for all $n\in \N$,
\[
\begin{split}
& \sum_{k\ge K} e^{\theta(p\vert
\alpha_k\vert)}\left(\sum_{l=0}^{m_k-1} \vert
c^n_{k,l}-c_{k,l}\vert (\vert
\alpha_k\vert+1)^{-(m_0+\cdots+m_{k-1}+l)}\right)\\ \le
&\frac{\varepsilon}{4}  \sum_{k\ge K}e^{\theta((p+1)\vert
\alpha_k\vert)}\left(\sum_{l=0}^{m_k-1} \vert
c^n_{k,l}-c_{k,l}\vert (\vert
\alpha_k\vert+1)^{-(m_0+\cdots+m_{k-1}+l)}\right)\\ \le
&\frac{\varepsilon}{4} \Vert c^n-c\Vert_{\theta,p+1}\le
\frac{\varepsilon}{4} (\Vert c^n\Vert_{\theta,p+1}+\Vert
c\Vert_{\theta,p+1})\le \frac{\varepsilon}{2}.\\
\end{split}
\]
Moreover, for a certain $N\in \N$ and for all $n\ge N$, we have
\[
\sum_{k=0}^{K-1} e^{{\theta}(p\vert
\alpha_k\vert)}\left(\sum_{l=0}^{m_k-1} \vert
{c^n}_{k,l}-c_{k,l}\vert (\vert
\alpha_k\vert+1)^{-(m_0+\cdots+m_{k-1}+l)}\right)\le
\frac{\varepsilon}{2}.
\]
Finally, for $n\ge N$, $\Vert c^n-c\Vert_{\theta,p}<\varepsilon$.

\end{proof}

From now on, we will identify ${\mathcal B}_{\theta}^{\prime}(V)$
with the space $\mathcal{C}_{\theta}(V)$. 
The next step is to prove the following lemma :

\begin{lemma}\label{fonc}

\noindent (i) $\alpha^t$ is a topological isomorphim onto its
image and $\Im \alpha^t= ($Ker $\alpha)^\circ$, the orthogonal
space of Ker $\alpha$ .

\noindent (ii) $\hbox{Ker }\alpha=\{T\star U,\ \ U\in
\mathcal{F}_{\theta}^{\prime}(\C)\}.$

\noindent (iii) (Ker $\alpha)^\circ$=Ker $T\star=\{f\in
\mathcal{F}_{\theta}(\C)\ \vert \ T\star f =0\}.$
\end{lemma}

\begin{proof}

\noindent (i)  From Proposition \ref{alpha},  $\alpha$ is a
surjective continuous linear map. Therefore, $\alpha^t$ is a
topological isomorphism onto its image and $\Im \alpha^t= ($Ker
$\alpha)^\circ$ (see \cite[Proposition 1.4.12]{Be-Ga}).

\noindent (ii) Recalling Remark \ref{Ker}, we have $$\hbox{Ker
}\alpha=\hbox{Ker }(\rho\circ
\mathcal{L})=\mathcal{L}^{-1}(\hbox{Ker}\rho)=\{T \star
\mathcal{L}^{-1}(g), \ g\in \mathcal{G}_{\theta}(\C)\})=\{T\star
U,\ \ U\in \mathcal{F}_{\theta}^{\prime}(\C)\}.$$

\noindent (iii) Let $f$ be an element of (Ker $\alpha)^\circ$. For
all $z\in \C$, $$(T\star f )(z)=<T,\tau_z f>=<T,\delta_z\star
f>=<T\star \delta_z,f>=0,$$ using the fact that $T\star  \delta_z
\in$ Ker $\alpha$.

Conversely, let $f\in  \mathcal{F}_{\theta}(\C)$ be such that
$T\star f=0$ and let  $U\in \mathcal{F}_{\theta}^{\prime}(\C)$. We
have $$<T\star U,f>=<U, T\star f>=0.$$ This shows that $f\in $(Ker
$\alpha)^\circ$ and concludes the proof of the lemma.

\end{proof}

Let us proceed with the proof of Theorem \ref{gen}.

(i) Let  $f \in \mathcal{F}_{\theta}(\C)$ be a $T$-mean-periodic
function, that is, $f\in$ Ker $T\star$. From Lemmas \ref{fonc} and
\ref{dualgen}, there is a unique sequence $c\in
\mathcal{C}_{\theta}(V)$ such that $f=\alpha^t(c)$.

For $z\in \C$, denoting by $\delta_z$ the Dirac measure at $z$, we
have
$$f(z)=<\delta_z,f>=<\delta_z,\alpha^t(c)>=<c,\alpha(\delta_z)>=<c,\Psi(\rho(g_z))>$$
where we have denoted by $g_z=\mathcal{L}(\delta_z)$, that is, the
function in $\mathcal{G}_{\theta}(\C)$ defined by $g_z(\xi)=e^{z
\xi}$.

Let us compute $\Psi(\rho(g_z))=b(z)=\{b_{k,l}(z)\}_{k,0\le l<
m_k}$, which is an element of ${\mathcal B}_{\theta}(V)$. By well
know formulas about Newton polynomials (See, for example
\cite[Definition 6.2.8]{Be-Ga}), we have, for $k\in \N$, and
denoting by $$\partial_j ^m=
\frac{1}{m!}\frac{\partial^m}{\partial \alpha_j^m},$$ for $0\le
l<m_k$,
$$b_{k,l}(z)=\partial_0^{m_0-1}\cdots\partial_{k-1}^{m_{k-1}-1}\partial_k^l\left(\sum_{j=0}^k
e^{z\alpha_j}\prod_{0\le n \le k,
n\not=j}(\alpha_j-\alpha_n)^{-1}\right)=\sum_{j=0}^k e^{z\alpha_j}
P_{k,j,l}(z),$$ where we have denoted by, for $j<k$,
\begin{equation}\label{poly1}
P_{k,j,l}(z)=\sum_{i=0}^{m_j-1}\frac{z^i}{i!}
\partial_i^{m_j-1-i}\left(\prod_{0\le n\le k-1,
n\not=j}(\alpha_j-\alpha_n)^{-m_n}(\alpha_j-\alpha_k)^{-(l+1)}\right)
\end{equation}
and
\begin{equation}\label{poly2}P_{k,k,l}(z)=\sum_{i=0}^l
\frac{z^i}{i!}\partial_k^{l-i}\left(\prod_{0\le n\le
k-1}(\alpha_k-\alpha_n)^{-m_n}\right).
\end{equation}
 Thus, $$f(z)=\sum_{k\ge 0}
\left(\sum_{l=0}^{m_k-1}c_{k,l} b_{k,l}(z)\right) =\sum_{k\ge
0}\left(\sum_{l=0}^{m_k-1}  c_{k,l} \sum_{j=0}^k e^{z\alpha_j}
P_{k,j,l}(z)\right)$$ and the equality \eqref{expgen} is established.
Let us now verify the convergence in
$\mathcal{F}_\theta(\C)$ of the series.

{\it Case where $\theta(x)=x$}. Here,
$\mathcal{F}_\theta(\C)=\mathcal{H}(\C)$. We have to verify that
the serie converges uniformly on every compact of $\C$. Let $p\in
\N^*$ and $z\in \C$, $\vert z\vert\le p$.

 We have, for all $\xi\in \C$, $\vert g_z(\xi)\vert=\vert e^{z\xi}\vert \le e^{p\vert \xi\vert}$, that is,
$$\Vert g_z\Vert_{\theta,p}\le 1.$$ Thus, by continuity of
$\Psi\circ \rho$, there exists $p'\in \N^*$ and $C_p>0$ such that
$$\Vert b(z)\Vert_{\theta,p'}\le C_p \Vert g_z\Vert_{\theta,p}\le
C_p.$$ For all $k\ge 0$, we have
\[
\begin{split}
\sum_{l=0}^{m_k-1} \vert c_{k,l}b_{k,l}(z)\vert &\le
\Vert b(z)\Vert_{\theta,p'} e^{\theta(p'\vert
\alpha_k\vert)}\sum_{l=0}^{m_k-1} \vert c_{k,l}\vert
(\vert \alpha_k\vert+1)^{-(m_0+\cdots+m_{k-1}+l)}\\
\end{split}
\]

We obtain $$\sup_{\vert z\vert \le p}\sum_{l=0}^{m_k-1}
\vert c_{k,l} b_{k,l}(z)\vert \le C_p e^{{\theta}(p'\vert
\alpha_k\vert)}\sum_{l=0}^{m_k-1}  \vert
c_{k,l}\vert(\vert
\alpha_k\vert+1)^{-(m_0+\cdots+m_{k-1}+l)}$$ Recalling that $c\in {\mathcal
C}_{\theta}^{\prime}(V)$, the right term is the general term of a
convergent serie, thus, the right-hand side of \eqref{expgen} is
 convergent in $\mathcal{F}_\theta(\C)$. Moreover,
$$\sup_{\vert z\vert \le p}\sum_{k\ge 0} \sum_{l=0}^{m_k-1}
\vert c_{k,l} b_{k,l}(z)\vert \le C_p \Vert c\Vert'_{\theta,p'}.$$

{\it Case where $\theta$ is a Young function}.
For any $p\in \N^*$,  observe that $$\Vert g_z\Vert_{\theta,p}\le
e^{\theta^*(\frac{1}{p}\vert z\vert)}.$$ Thus, by continuity of
$\Psi\circ \rho$, there exists $p'\in \N^*$ and $C_p>0$ such that
$$\Vert b(z)\Vert_{\theta,p'}\le C_p \Vert g_z\Vert_p\le C_p
e^{\theta^*(\frac{1}{p}\vert z\vert)}.$$ For all $k\in \N$ and
$z\in \C$, we have
\[
 \sum_{l=0}^{m_k-1} \vert c_{k,l}b_{k,l}(z) \vert  \le
\Vert b(z)\Vert_{\theta,p'}e^{-\theta(p'\vert
\alpha_k\vert)}\sum_{l=0}^{m_k-1} \vert c_{k,l}\vert
(\vert \alpha_k\vert+1)^{m_0+\cdots+m_{k-1}+l}\\
\]
We obtain 
\[ \sup_{z\in \C} \sum_{l=0}^{m_k-1} \vert c_{k,l} b_{k,l}(z)
\vert e^{-\theta^*(\frac{1}{p}\vert z\vert)}\le C_p
e^{-{\theta}(p'\vert
\alpha_k\vert)}\sum_{l=0}^{m_k-1}\vert c_{k,l}\vert (\vert
\alpha_k\vert+1)^{m_0+\cdots+m_{k-1}+l} .
\]
 As in the previous case, we
deduce that the right-hand side of \eqref{expgen} is absolutely
convergent in $\mathcal{F}_\theta(\C)$. Moreover, 
\[
\sup_{z\in
\C}\sum_{k\ge 0}\sum_{l=0}^{m_k-1} \vert  c_{k,l} b_{k,l}(z)\vert
\le C_p \Vert c\Vert'_{\theta,p'}.
\]
 In order to find an explicit
formula for the coefficients $c_{n,l}$, consider the elements
$B^{k,l}$ of $\mathcal{B}_{\theta}(\C)$ defined by \eqref{Bnl} and
observe that,  by the definition of the Newton polynomials (see
\eqref{newton}) with respect to the coefficients of $B^{k,l}$, for
all $q\ge k$, we have $$Q_q(\xi)=(\xi-\alpha_k)^l\prod_{l=0}^{k-1}
(\xi-\alpha_l)^{m_l}$$
and for $q<k$, $Q_q=0$. We readily deduce that
$\alpha(S_{k,l})=\Psi\circ\rho\circ\mathcal{L}(S_{k,l})=B^{k,l}.$

Now, for all $k\in \N$ and $0\le l<m_k$,
$$<S_{k,l},f>=<S_{k,l},\alpha^t(c)>=<\alpha(S_{k,l}),c>=<B^{k,l},c>=c_{k,l}.$$

(ii) The converse part is easily deduced from the estimates in the proof of (i) and
Lemma \ref{monomials}.

\end{section}

\begin{section}{Case where $V$ is an interpolating variety.}\label{intcase}

 \begin{definition}\label{intvar}
 We say that $V$ is an interpolating variety for ${\mathcal G}_{\theta}(\C)$ if, for any doubly indexed sequence $a=\{a_{k,l}\}_{k\in \N, 0\le l<m_k}$  such that, for a certain $m>0$,
$$\sup_{k\in \N}  \sum_{l=0}^{m_k-1}\vert a_{k,l}\vert
e^{-{\theta}(m \vert \alpha_k\vert)}<+\infty,$$ there exists a function
$g\in {\mathcal G}_{\theta}(\C)$ such that, for all $k$ and all
$0\le l<m_k-1$, $$\frac{g^{l}(\alpha_k)}{l!}=a_{k,l}.$$
\end{definition}
We assume from now on that $V$ is an interpolating variety for
${\mathcal G}_{\theta}(\C)$. Then we have the following result :

\begin{theorem}\label{int}

\noindent (i) Any $T$-mean-periodic function $f\in
\mathcal{F}_{\theta}(\C)$ admits the following expansion as a
 convergent series in  $\mathcal{F}_\theta(\C)$
\begin{equation}\label{expint}
f(z)=\sum_{k\ge 0}  e^{z\alpha_k}\sum_{l=0}^{m_k-1}
d_{k,l}\frac{z^l}{l!} ,
\end{equation}
 where the coefficients $a_{k,l}$ verify the following estimate :
\begin{equation}\label{dkj}
\sum_{k\ge 0} e^{{\theta}(m\vert
\alpha_k\vert)}\left(\sum_{l=0}^{m_k-1} \vert d_{k,l}\vert \right)
<+\infty
\end{equation}
for every $m>0$. Moreover, for all $k\in \N$ and $0\le l<m_k$,
we have the equality $$d_{k,l}=<T_{k,l},f>$$ where $T_{k,l}\in
\mathcal{F}_{\theta}^{\prime}(\C)$ is defined by
$$\mathcal{L}(T_{k,l})(\xi)=\frac{m_k!}{\Phi^{(m_k)}(\alpha_k)}\frac{\Phi(\xi)}{(\xi-\alpha_k)^{m_k-l}},$$

\noindent (ii) Conversely, any such series whose coefficients
$d_{k,l}$ satisfy these estimate \eqref{dkj} converges in
$\mathcal{F}_\theta(\C)$ to a function $f$ solving the equation
\eqref{equ}.

\end{theorem}

Note that $\mathcal{L}(T_{k,l})\in \mathcal{G}_{\theta}(\C)$ by
Proposition \ref{div}.
\begin{remark}
In the case where ${\theta}(x)=x$, this Theorem 5.2 is also a consequence of
\cite[Theorem 6.2.6.]{Be-Ga}.
\end{remark}

We will denote by ${\mathcal A}_{\theta,m}(V)$ the space of all
doubly indexed sequences of complex numbers $a=\{a_{k,l}\}_{k\in
\N, 0\le l<m_k}$  such that 
\begin{equation}\label{norm}
\Vert  a
\Vert_{\theta,m}:=\sup_{k\in \N}  \sum_{l=0}^{m_k-1}\vert
a_{k,l}\vert e^{-{\theta}(m \vert \alpha _k\vert)}<+\infty
\end{equation} and
$$ \mathcal{A}_{\theta}(V)=\cup_{p\in
N^*}\mathcal{A}_{\theta,p}(V)$$ endowed with the strict inductive
limit of Banach spaces.

We define the linear map $$\alpha= \rho \circ \mathcal{L} :
\mathcal{F}_{\theta}^{\prime}(\C)\rightarrow
\mathcal{A}_{\theta}(V).$$

\begin{proposition}\label{alphaint}
The map $\alpha$  is continuous and surjective.
\end{proposition}

\begin{proof}

It is sufficient to show that the map $\rho : {\mathcal
G}_{\theta}(\C)\rightarrow \mathcal{A}_{\theta}(V)$ is surjective
and continuous. The surjectivity follows from the fact that $V$ is
an interpolating variety.

In order to show the continuity, let $g\in  {\mathcal
G}_{\theta,p}(\C)$ and let $z\in \C$. By the Cauchy estimates
applied to the disc of center $z$ and radius $2$, for all $l\in
\N$, $$\left \vert \frac{g^{l}(z)}{l!}\right \vert \le
\frac{1}{2^l}\sup_{\vert \xi-z\vert\le 2}\vert g(\xi)\vert.$$ For
$\vert \xi-z\vert \le 2$, we have $$\vert g(\xi)\vert \le \Vert
g\Vert_{\theta,p} e^{\theta(p\vert \xi\vert)}\le \Vert
g\Vert_{\theta,p} e^{\theta(2p+p\vert z\vert)}\le \Vert
g\Vert_{\theta,p}e^{1/2\theta(4p)}e^{1/2\theta(2p\vert z\vert)}$$
by convexity of $\theta$. Thus, $$\sum_{l=0}^\infty \left\vert
\frac{g^{l}(z)}{l!}\right\vert \le 2 \Vert
g\Vert_{\theta,p}e^{1/2\theta(2p)}e^{\theta(2p\vert z\vert)}.$$ In
particular, we deduce that $\rho(g)\in \mathcal{A}_{\theta,2p}(V)$
and that $$\Vert \rho(g)\Vert_{\theta,2p}\le 2 \Vert
g\Vert_{\theta,p}e^{1/2\theta(2p)}.$$ The continuity of $\rho$
follows from the last inequality. 
\end{proof}
\begin{remark}
By a standard result about interpolating varieties (see \cite[chapter 2]{Be-Ga}),  the multiplicities verify, for  certain constants $A,m>0$,  
$$m_k\le A e^{\theta(m\vert \alpha_k\vert)},\ \ \forall k\in \N.$$ 
 Consequently, we may replace  the norm given by  (\ref{norm}) by the following
$$\Vert  a
\Vert_{\theta,m}:=\sup_{k\in \N}  \sup_{0\le l<m_k-1}\vert
a_{k,l}\vert e^{-{\theta}(m \vert \alpha _k\vert)}$$
in the definition of ${\mathcal A}_{\theta}(V)$.
\end{remark}

\begin{lemma}\label {dualint}
The space ${\mathcal A}_{\theta}^{\prime}(V)$ is topologically
isomorphic to the space $\mathcal{D}_{\theta}(V)=\cap_{p\in
\N^*}\mathcal{D}_{\theta,p}(V)$ endowed with the projective limit
topology, where, for all $p$, $\mathcal{D}_{\theta,p}(V)$ is the
Banach space of the sequences $d=\{d_{k,l}\}_{k, 0\le l<m_k}$ such
that
\begin{equation}\label{dkjp}
\Vert d\Vert'_{\theta,p} :=\sum_{k\ge 0} e^{{\theta}(p\vert
\alpha_k\vert)}\left(\sum_{l=0}^{m_k-1} \vert d_{k,l}\vert\right) <+\infty.
\end{equation}
Moreover, $\mathcal{D}_{\theta}(V)$ is a Fr\'echet-Schwartz space.
\end{lemma}
In view of the preceding remark, the proof is similar to  the one of Lemma \ref{dualgen}.
We are now ready to prove Theorem \ref{int}.
From Lemma \ref{fonc} (which is still valid with the new definition of $\rho$) and Lemma \ref {dualint}, any $T$-mean-periodic function $f$ is the
image by $\alpha^t$ of a unique $d\in
\mathcal{A}_{\theta}^{\prime}(V)$. We have, for all $z\in \C$,
$$f(z)=<\delta_z,f>=<\delta_z,
\alpha^t(d)>=<d,\alpha(\delta_z)>=<d,\rho(g_z)>=\sum_{k\ge 0}
e^{z\alpha_k}\sum_{l=0}^{m_k-1}\frac{z^l}{l!}d_{k,l}.$$ To compute
the coefficients $d_{k,l}$ :
$$<T_{k,l},f>=<T_{k,l},\alpha^t(d)>=<d,
\alpha(T_{k,l})>=d_{k,l}.$$ The last equality follows from the
observation that $$\alpha(T_{k,l})=B^{k,l}.$$

The rest of the proof is similar to the one of Theorem \ref{gen}.

Let us recall some results about interpolating varieties that
enables one to determine whether $V$ is interpolating or not. We first give a known analytic characterization (see \cite{Be-Ta} or \cite{Be-Ga}). The spaces of entire functions considered are slightly
different, but is clear how to adapt these results to our spaces.
\begin{theorem}
 $V$ is an interpolating variety for
$\mathcal{G}_{\theta}(\C)$ if and only if, there are constants
$\varepsilon>0$ and $m>0$ such that, for all $k$,
 $$\left \vert \frac{\Phi^{m_k}(z)}{m_k!}\right \vert\ge
\varepsilon e^{-\theta(m\vert \alpha_k\vert)}.$$
\end{theorem}

We also give a known  geometric characterization  (see \cite[Corollary 4.8]{Be-Li}  or \cite[Theorem 1.8]{Ou1})
in terms of the distribution of the points
$\{(\alpha_k,m_k)\}_k$.

Define the counting function and the integrated counting function
:

\begin{definition}\label{count}
For $z\in \C$ and $r>0$,
\[
n(z,r)=\sum\limits_{|z-\alpha_k|\le r} m_k,
\]

\[
N(z,r)=\int_0^r\frac{n(z,t)-n(z,0)}t\, dt + n(z,0)\ln
r=\sum_{0<\vert z- \alpha_k\vert\le r}m_k \ln {r\over \vert
z-\alpha_k\vert}+n(z,0)\ln r .
\]
\end{definition}

\begin{theorem}

$V$ is an interpolating variety for $\mathcal{G}_{\theta}(\C)$ if
and only if conditions

\begin{equation}\label{N(0,R)}
\exists A>0,\  \exists m>0\ \ \forall R>0,\ \ \ N(0,R) \le A+
{\theta}(m R )
\end{equation}

and

\begin{equation}\label{N(z,z)}
\exists A>0,\  \exists m>0\ \ \forall k\in \N,\ \ \ N(\alpha_k,\vert
\alpha_k\vert) \le A+{\theta}(m \vert \alpha_k\vert)
\end{equation}
hold.

\end{theorem}
 Actually, in this paper, since $V=\Phi^{-1}(0)$ and $\Phi \in\mathcal{G}_{\theta}(\C)$, condition \eqref{N(0,R)} is necessarily verified (see, for example, \cite[Theorem 1.13]{Ou2}). Thus, $V$ is an interpolating variety if and only if condition \eqref{N(z,z)} holds.

 \begin{remark}
We can obtain Theorem \ref{int}  as a corollary of Theorem
\ref{gen}, using the density condition \eqref{N(z,z)}. This second
proof is rather technical, we will skip it here. Let us just give
the correspondence between the coefficients $c_{k,l}$ and
$d_{k,l}$ :

\begin{equation}
\begin{split}
d_{k,l}= & \sum_{i=l}^{m_k-1}c_{k,i}
\,\partial_k^{i-l}\left(\prod_{0\le n\le
k-1}(\alpha_k-\alpha_n)^{-m_n}\right)\\ & +\sum_{j=k+1}^\infty
\sum_{i=0}^{m_j-1}c_{j,i}\,\partial_k^{m_k-1-l}\left(\prod_{0\le
n\le j-1,
n\not=k}(\alpha_k-\alpha_n)^{-m_n}(\alpha_k-\alpha_j)^{-(i+1)}\right),
\end{split}
\end{equation}
the convergence of the second sum being a consequence of
conditions  \eqref{N(z,z)} and  \eqref{N(0,R)}.

In the case where all $m_k=1$, we have $$d_k=\sum_{j=k}^\infty c_j
\prod_{0\le n\le j,n\not=k} (\alpha_k-\alpha_n)^{-1}.$$
\end{remark}

\end{section}

\bibliographystyle{plain}

\begin{thebibliography}{10}

\bibitem{Be-Ga1}
C.~A. Berenstein and R.~Gay.
\newblock Complex variables, an introduction,
\newblock Springer-Verlag, New York, 1991.













\bibitem{Be-Ga}
C.~A. Berenstein and R.~Gay.
\newblock  Complex analysis and special topics in harmonic analysis,
\newblock Springer-Verlag, New York, 1995.

\bibitem{Be-Li}
C.~A. Berenstein and B.~Q. Li.
\newblock {\em Interpolating varieties for spaces of meromorphic functions},
\newblock {J. Geom. Anal.} {\bf 5 (1)} (1995), 1-48 .




\bibitem{Be-St} C.~A. Berenstein and D.~C. Struppa.
\newblock {\em Dirichlet series and convolution equations}
\newblock {Publ. RIMS, Kyoto Univ.} {\bf 24} (1988), 783-810.

\bibitem{Be-St1} C.~A. Berenstein and D.~C. Struppa.
\newblock {\em Complex analysis and
convolution equations},
\newblock {Itogi Nauki Tekh., Ser. Sovrem. Probl.
Mat., Fundam. Napravleniya.} 54 (1989), 5-111 ; English transl.
Several complex variables. V: Complex analysis in partial
differential equations and mathematical physics, Encycl. Math.
Sci. {\bf 54}  (1993), 1-108.

\bibitem{Be-Ta}
C.~A. Berenstein and B.~A. Taylor.
\newblock A new look at interpolation theory for entire functions of one
  variable,
\newblock {\em Adv. in Math.} {\bf 33 (2)} (1979), 109-143, .



\bibitem{De}
J.~Delsarte.
\newblock {\em Les fonctions moyennes-p\'eriodiques},
\newblock {J. Math. Pures et Appl.} {\bf 14} (1935), 403-453.




\bibitem{GHOR}
R.~Gannoun, R.~Hachaichi, H.~Ouerdiane, and A.~Rezgui.
\newblock {\em Un th\'eor\`eme de dualit\'e entre espaces de fonctions holomorphes
  \`a croissance exponentielle},
\newblock {J. Funct. Anal.} {\bf 171 (1)} (2000), 1-14.

\bibitem{Is-Ke}
E.~Isaacson and H.~B Keller.
\newblock Analysis of numerical methods,
\newblock Dover Publications Inc., New York, 1994.
\newblock Corrected reprint of the 1966 original [Wiley, New York].

\bibitem{Kr}
M.~A. Krasnosel{\cprime}ski{\u\i} and Ja.~B. Ruticki{\u\i}.
\newblock {\em Convex functions and {O}rlicz spaces}.
\newblock Translated from the first Russian edition by Leo F. Boron. P.
  Noordhoff Ltd., Groningen, 1961.



\bibitem{Ou1}
M.~Ouna{\"{\i}}es.
\newblock {\em Geometric Conditions for Interpolation in Weighted Spaces of Entire Functions}.
\newblock {J. Geom. Anal.} {\bf 17 (4)} 2008, 701-716.



\bibitem{Ou2}
M.~Ouna{\"{\i}}es.
\newblock {\em Interpolation by entire functions with growth conditions}.
\newblock {Michigan Math. J.} {\bf 56 (1)} 2008.


\bibitem{Sc}
L.~Schwartz.
\newblock {\em Espaces nucl\'eaires}.
\newblock {Seminaire Schwartz} {\bf 1 (Expos\'e 17)} (1935-54).

\end{thebibliography}
\def\cprime{$'$}

\end{document}